\documentclass[12pt]{amsart}

\usepackage{amsmath}
\usepackage{amsthm}
\usepackage{amsfonts}
\usepackage{comment}
\usepackage{amssymb}
\usepackage{eufrak}
\usepackage{comment}
\usepackage{hyperref}
\usepackage{mathtools}

\usepackage[margin=1.25in]{geometry}
\usepackage{mathrsfs,manfnt,enumitem}
\usepackage[super]{nth}
\usepackage{setspace}
\usepackage{verbatim}
\usepackage{calc,color}

\usepackage{lineno}
\usepackage{amsxtra,bm,fancybox,fancyhdr}
\usepackage{eucal}
\usepackage[all]{xy}
\usepackage{graphicx}
\usepackage{hyperref}
\usepackage{mathtools}

\newtheorem{theorem}{Theorem}
\newtheorem{proposition}[theorem]{Proposition}
\newtheorem{lemma}[theorem]{Lemma}

\newtheorem{conjecture}[theorem]{Conjecture}

\theoremstyle{definition}
\newtheorem{defn}[theorem]{Definition}
\newtheorem{example}[theorem]{Example}

\newtheorem*{theorem*}{Theorem}
\newtheorem*{proposition*}{Proposition}
\newtheorem*{lemma*}{Lemma}

\theoremstyle{remark}
\newtheorem*{remark}{Remark}

\newcommand{\id}{\operatorname{id}}
\newcommand{\defeq}{\vcentcolon=}

\newcommand\nc{\newcommand}
\nc{\on}{\operatorname}
\nc\renc{\renewcommand}
\nc{\BR}{\mathbb R}
\nc{\BC}{\mathbb C}
\nc{\BQ}{\mathbb Q}
\nc{\BZ}{\mathbb Z}
\nc{\BN}{\mathbb N}
\nc{\BP}{\mathbb P}
\nc{\Hom}{\on{Hom}}
\nc{\wt}{\widetilde}
\nc{\vspan}{\on{span}}
\nc{\ord}{\on{ord}}
\nc{\im}{\on{im}}
\nc{\Mat}{\on{Mat}}
\nc{\can}{\on{can}}
\nc{\coker}{\on{coker}}
\nc{\ev}{\on{ev}}
\nc{\Tr}{\on{Tr}}
\nc{\End}{\on{End}}
\nc{\swap}{\on{swap}}
\nc{\Set}{\on{Set}}
\nc{\bC}{{\bm C}}
\nc{\bc}{{\bm c}}
\nc{\bD}{{\bm D}}
\nc{\bd}{{\bm d}}
\nc{\bE}{{\bm E}}
\nc{\be}{{\bm e}}
\nc{\bF}{{\bm F}}
\nc{\bff}{{\bm f}}

\nc{\adj}{\on{adj}}
\nc{\tensor}[3]{#1 \underset{#2}\otimes #3}
\nc{\Nat}{\on{Nat}}
\nc{\op}{\on{op}}
\nc{\Funct}{\on{Funct}}
\nc{\Ob}{\on{Ob}}
\nc{\fR}{\mathfrak{R}}
\nc{\Vect}{\on{Vect}}
\nc{\ns}{\on{non-spec}}
\nc{\bra}[1]{\left\langle #1 \right|}
\nc{\ket}[1]{\left| #1 \right\rangle}
\nc{\braket}[2]{\left\langle #1 \mid #2 \right\rangle}
\nc{\ol}{\overline}
\nc{\ul}{\underline}
\nc{\w}{\omega}
\nc{\nlog}{\on{nlog}}
\nc{\Aut}{\on{Aut}}
\nc{\Gal}{\on{Gal}}
\nc{\ksep}{\on{sep}}
\nc{\low}{\on{low}}
\nc{\Stab}{\on{Stab}}

 \allowdisplaybreaks[1]


\begin{document}

\title[Arboreal Galois Representations of Rational Functions]{On Arboreal Galois Representations of \\Rational Functions}

\author[Ashvin A. Swaminathan]{Ashvin A. Swaminathan}\address{ Department of Mathematics, Harvard College, 1 Oxford Street, \mbox{Cambridge, MA 02138}}

\email{\href{mailto:aaswaminathan@college.harvard.edu}{{\tt aaswaminathan@college.harvard.edu}}
}

\maketitle
\vspace*{-0.25in}
\begin{abstract}
\noindent The action of the absolute Galois group $\Gal(K^{\ksep}/K)$ of a global field $K$ on a tree $T(\phi, \alpha)$ of iterated preimages of $\alpha \in \BP^1(K)$ under $\phi \in K(x)$ with $\deg(\phi) \geq 2$ induces a homomorphism $\rho: \Gal(K^{\ksep}/K) \to \Aut(T(\phi, \alpha))$, which is called an \emph{arboreal Galois representation}. In this paper, we address a number of questions posed by Jones and Manes (\cite{jone,JM}) about the size of the group $G(\phi,\alpha) \defeq \im \rho = \underset{\leftarrow n}\lim\Gal(K(\phi^{-n}(\alpha))/K)$. Specifically, we consider two cases for the pair $(\phi, \alpha)$: (1) $\phi$ is such that the sequence $\{a_n\}$ defined by $a_0  = \alpha$ and $a_n = \phi(a_{n-1})$ is periodic, and (2) $\phi$ commutes with a nontrivial M\"{o}bius transformation \mbox{that fixes $\alpha$.}

 In the first case, we resolve a question posed by Jones (\cite{jone}) about the size of $G(\phi, \alpha)$, and taking $K = \BQ$, we describe the Galois groups of iterates of polynomials $\phi \in \BZ[x]$ that have the form $\phi(x) = x^2 + kx$ or $\phi(x) = x^2 - (k+1)x + k$. When $K = \BQ$ and $\phi \in \BZ[x]$, arboreal Galois representations are a useful tool for studying the arithmetic dynamics of $\phi$. In the case of $\phi(x) = x^2 + kx$ for $k \in \BZ$, we employ a result of Jones (\cite{jone2}) regarding the size of the group $G(\psi, 0)$, where $\psi(x) = x^2 - kx + k$, to obtain a zero-density result for primes dividing terms of the sequence $\{a_n\}$ defined by $a_0 \in \BZ$ and $a_n = \phi(a_{n-1})$.

 In the second case, we resolve a conjecture of Jones (\cite{jone}) about the size of a certain subgroup $C(\phi, \alpha) \subset \Aut(T(\phi, \alpha))$ that contains $G(\phi, \alpha)$, and we present progress toward the proof of a conjecture of Jones and Manes (\cite{JM}) concerning the size of $G(\phi, \alpha)$ as a subgroup of $C(\phi, \alpha)$.
\end{abstract}

\section{Introduction}\label{intro}

\subsection{Background and Definitions}

\noindent Understanding which primes divide the Fermat numbers (i.e. numbers of the form $2^{2^n} + 1$) has been a question of interest for nearly four centuries. Progress toward answering related questions has been made with the tools of arithmetic dynamics, a field whose objective is to study the behavior of \emph{dynamical sequences} $\{a_n\}$ that are defined by $a_n = \phi(a_{n-1})$ for some choice of a zeroth term $a_0 \in \BZ$ and a polynomial $\phi \in \BZ[x]$. In the case of the Fermat numbers, the relevance of arithmetic dynamics is evident, because the sequence $\{a_n\}$ of Fermat numbers can be expressed in recursive form as $a_0 = 3$ and $a_n = \phi(a_{n-1})$, where $\phi \in \BZ[x]$ is given by $\phi(x) = (x-1)^2 + 1 = x^2 - 2x + 2$. Although it is difficult to give explicit descriptions of the primes that divide the terms of dynamical sequences, we may still be able to obtain qualitative results regarding the distribution of these prime factors in the set of all primes. In particular, given a choice of $a_0 \in \BZ$ and $\phi \in \BZ[x]$, we may attempt to determine the \emph{natural density} $\on{nd}(\phi, a_0)$ of the set $$P_\phi(a_0) = \{\text{prime } p : p \mid a_n \text{ for at least one } n \geq 0\}$$ as a subset of the set of all primes, where $\on{nd}(\phi, a_0)$ is defined by the limit
$$\on{nd}(\phi,a_0) = \limsup_{x \rightarrow \infty} \frac{|\{p \in P_\phi(a_0) : p \leq x\}|}{|\{\text{prime } p \leq x \}|}.$$
In~\cite{odo1} and~\cite{odo2}, Odoni showed that $\on{nd}(\phi, a_0) = 0$ when the Galois groups $G_n(\phi)$ of iterates $\phi^n$ of $\phi$ satisfy the equality
$$\lim_{n \rightarrow \infty} \frac{|\{g \in G_n(\phi) : g \text{ fixes at least one root of } \phi^n\}|}{|G_n(\phi)|} = 0.$$
In the case of \emph{Sylvester's sequence}, which is the sequence $\{a_n\}$ obtained by taking $a_0 = 2$ and $\phi(x) = x^2 - x + 1$, Odoni proved that $\on{nd}(\phi,2) = 0$ by showing that for each $n$ we have $G_n(\phi) \simeq \Aut(T_n)$, where $T_n$ is the complete, rooted binary tree of height $n$. To obtain this result for Sylvester's sequence, Odoni expressed $T_n$ as the tree $T_n(\phi,0)$ of iterated preimages of $0$ under $\phi$ (our notation for a preimage tree of height $n$ is $T_n($``function'',``root''$)$). More precisely, taking $\phi^0(x) = x$ and $\phi^{-n}(x) = \phi^{-1}(\phi^{-(n-1)}(x))$ for each $n \geq 1$, we can construct a binary tree $T(\phi,0)$ of infinite height rooted at $0$ such that the nodes on level $i$ are the elements of $\phi^{-i}(0)$ and such that for each $y \in \phi^{-i}(0)$, the parent of $y$ in level $i-1$ is $\phi(y) \in \phi^{-(i-1)}(0)$. We may then take $T_n(\phi, 0)$ to be the rooted binary subtree of $T(\phi, 0)$ obtained by deleting all levels of $T(\phi, 0)$ beyond level $n$. (Notice that $T(\phi,0)$ is a complete tree if and only if the discriminant of each iterate of $\phi$ does not vanish; when $\phi(x) = x^2 - x + 1$, the tree $T(\phi,0)$ is in fact complete.) As a consequence of Odoni's success with Sylvester's sequence, it is natural to study maps of the form $G_n(\phi) \to \Aut(T_n(\phi,0))$, which are called \emph{arboreal Galois representations}, for various choices of $\phi \in \BZ[x]$.

In~\cite{jone} and~\cite{JM}, Jones and Manes consider a more general situation. Let $K$ be a global field (i.e., an algebraic number field or function field over a finite field), let $\phi \in K(x)$ be a rational function of degree $d \geq 2$ (the case of $d = 1$ is well-studied), and let $\alpha \in \BP^1(K)$ (the projective line of the field $K$). Furthermore, consider only those pairs $(\phi, \alpha)$ such that the equation $\phi^n(x) = \alpha$ has $d^n$ distinct solutions for each $n$, where $\phi^n$ denotes the $n^\mathrm{th}$ iterate of $\phi$ (this condition is easily verified by showing that the orbit of each critical point of $\phi$ avoids $\alpha$). Then the tree $T(\phi, \alpha)$ of iterated preimages of $\alpha$ under $\phi$ is complete and $d$-ary, and its nodes are elements of the separable closure $K^{\ksep}$ of the field $K$. We observe that the elements of the absolute Galois group $\Gal(K^{\ksep}/K)$ induce tree automorphisms of $T(\phi, \alpha)$, so we obtain a homomorphism
$$\rho : \Gal(K^{\ksep}/K) \to \Aut(T(\phi, \alpha)),$$
which is again an arboreal Galois representation. Define $G(\phi, \alpha) \defeq \im \rho$. Notice that $G_n(\phi, \alpha) = \Gal(K(\phi^{-n}(\alpha))/K)$ is the quotient of $G(\phi, \alpha)$ by the equivalence relation $a \sim b$ if $a$ and $b$ act in the same way on $T_n(\phi, \alpha)$. Thus, we may view $G(\phi, \alpha)$ as the limit of the inverse system $\{G_n(\phi, \alpha)\}$. It suffices to take $\alpha = 0$: if $g(x) = x + \alpha$, then we have the equality of Galois groups
$$G_n(\phi, \alpha) = G_n(g^{-1} \circ \phi \circ  g, 0),$$
so we may replace $\phi$ with $g^{-1} \circ \phi \circ g$ and $\alpha$ with $0$.
We then write $T(\phi, 0)$ as $T(\phi)$ and $G(\phi, 0)$ as $G(\phi)$ for short. In light of Odoni's result for Sylvester's sequence, it is natural to try to characterize the structure of $G(\phi)$ as a subgroup of $\Aut(T(\phi))$.

\subsection{Main Results}\label{mainresults}

\noindent In this paper, we extend work of Jones and Manes (see~\cite{jone} and~\cite{JM}) on describing the group $G(\phi)$ for various choices of rational functions $\phi \in K(x)$. Particularly, we consider the following two cases for which the structure of $G(\phi)$ is not fully understood.

\medskip
\noindent {\bf Problem (1):} The sequence $\{a_n\}$ defined by $a_0 = 0$ and $a_n = \phi(a_{n-1})$ is periodic. (Actually, we want to study this problem for arbitrary $a_0 \in K$. Nonetheless, if $a_0 \neq 0$, we may conjugate $\phi$ by $g(x) = x + a_0$ to set $a_0 = 0$.)\\
\noindent {\bf Problem (2):} The function $\phi$ commutes with a nontrivial M\"{o}bius transformation that fixes the root $\alpha = 0$ of the tree $T(\phi)$.
\medskip

As exemplified by Odoni's work, understanding the structure of $G(\phi)$ may be useful for obtaining zero-density results for $P_\phi(a_0)$. In~\cite{jone}, Jones describes a stochastic method that can in some cases translate results about the size of $G(\phi)$ into zero-density results for $P_\phi(a_0)$. One key step in Jones' stochastic method is to find ``natural'' subgroups $N(\phi) \subset \Aut(T(\phi))$ that satisfy $G(\phi) \subset N(\phi)$ and $[N(\phi) : G(\phi)] < \infty$. By ``natural,'' we mean subgroups whose action on $T(\phi)$ can be stated explicitly in the language of tree automorphisms, without any reference to elements of $G(\phi)$. In the case of Sylvester's sequence, the desired subgroup $N(\phi)$ happens to be $\Aut(T(\phi))$ itself. But it is easy to show that in the two problems enumerated above, $[\Aut(T(\phi)) : G(\phi)] = \infty$. Consequently, one of our objectives in describing $G(\phi)$ is to find and test candidates for the group $N(\phi)$ when $\phi$ belongs to either of the two problems enumerated above.

The rest of the paper is organized as follows. Section~\ref{period} discusses our results regarding Problem (1) above. The primary results proved in Section~\ref{period} are as follows. We first show that $G(\phi)$ is a small subgroup of the stabilizer $\Stab(\phi) \subset \Aut(T(\phi))$ of the branch whose nodes are precisely the terms of the sequence $\{a_n\}$, thereby resolving a question of Jones (Question 3.4 on p. 18 of~\cite{jone}).

\begin{proposition*}[Proposition~\ref{prop3}]
For every $\phi \in K(x)$ such that $\phi^m(0) = 0$ for some positive integer $m$,  $[\Stab(\phi) : G(\phi)]$ is necessarily infinite.
\end{proposition*}

Let $S(\phi) \subset \Stab(\phi)$ be the group of automorphisms of $T(\phi)$ that act in the same way on the subtrees rooted at $a_r$ and $a_s$ whenever $r \equiv s$ modulo the period of the sequence. For the family $\phi \in \BZ[x]$ defined by $\phi(x) = x^2 + kx$, we obtain the following characterization of $G(\phi)$.

\begin{theorem*}[Theorems~\ref{thm8} and~\ref{thm9}]
If $k \in \BZ \setminus \{-2, 0, 2, 4\}$ and $\phi(x) = x^2 + kx$, then $[S(\phi) : G(\phi)] < \infty$. On the other hand, if $k \in \{-2,0,2,4\}$, then $[S(\phi) : G(\phi)] = \infty$, but for each $k$ we can compute a valid \mbox{choice for $N(\phi)$}.
\end{theorem*}

We obtain a similar but more restricted result in the case when $\phi(x) = x^2 - (k+1)x + k$.

Sections~\ref{trees} and~\ref{squares} treat our results regarding Problem (2) above. In Section~\ref{trees}, we first prove a general lemma about centralizers of subgroups of $\Aut(T)$ for infinite, rooted $d$-ary trees $T$ (see Lemma~\ref{cor1} for the precise statement). Using this lemma, we resolve Conjecture 3.5 of~\cite{jone}, which is proven as Theorem~\ref{thm5} in the present paper. Define $C(\phi) \subset \Aut(T(\phi))$ to be the centralizer of the group $A(\phi)$ of M\"{o}bius transformations that commute with $\phi$ and fix $0$ (the group $A(\phi)$ embeds into $\Aut(T(\phi))$). Then following theorem tells us that $C(\phi)$ is in general a small subgroup of $\Aut(T(\phi))$.

\begin{theorem*}[Theorem~\ref{thm5}]
When $A(\phi)$ is nontrivial we have $[\Aut(T(\phi)) : C(\phi)] = \infty$.
\end{theorem*}

In Section~\ref{squares}, we present progress made toward a proof of Conjecture 1.3 of~\cite{JM}, a restricted version of which states that $[C(\phi) : G(\phi)] < \infty$ when $K = \BQ$ and $\phi$ is a quadratic rational function of the form $\phi(x) = \frac{k_0(x^2 + 1)}{x}$ where $k_0 \in \BZ \setminus \{0\}$. In~\cite{JM} it is shown that $[C(\phi) : G(\phi)]$ is finite when a certain sequence $\{\delta_n(k_0) : n = 2, 3, \ldots\}$ contains no squares. Our first result extends a lower bound from~\cite{JM} on $k_0$-values for which $[C(\phi) : G(\phi)] < \infty$.

\begin{theorem*}[Theorem~\ref{thm6}]
When $1 \leq |k_0| \leq 10^6$, the sequence $\{\delta_n(k_0)\}$ contains no squares. Thus, when $\phi(x) = \frac{k_0(x^2 + 1)}{x}$ and $1 \leq |k_0| \leq 10^6$, we have $[C(\phi) : G(\phi)] < \infty$.
\end{theorem*}

The lower bound of $10^6$ on $k_0$-values was proven by reducing $\delta_n(k_0)$ modulo various primes. Our final results in Section~\ref{squares} explore an algebraic method of showing that $\delta_n(k_0)$ is not a square. In particular, we prove the following theorem, which provides significant evidence toward the conjecture that $\delta_n(k_0)$ is not a square for all integers $n \geq 2$ and $k_0 \neq 0$.

\begin{theorem*}[Theorem~\ref{thm4}]
Fix $n \geq 2$. Then $\delta_n(k_0)$ is a square for at most finitely many nonzero integers $k_0$.
\end{theorem*}

\section{When the Sequence $\{a_n\}$ is Periodic}\label{period}

In this section, we discuss our results regarding Problem (1) of Section~\ref{mainresults}.

\subsection{The Size of the Stabilizer of $\{a_n\}$}

\noindent Suppose that $\phi^m(0) = 0$ for some positive integer $m$. Consider the sequence $\{a_n\}$ defined by $a_0 = 0 \text{ and } a_n = \phi(a_{n-1})$. Clearly $\{a_n\}$ is periodic, since $a_m = 0 = a_0$, and its elements belong to $K$. Also, notice that $a_i \in \phi^{i-m}(0)$ for each $i \in \{0, \dots, m-1\}$. Thus, in the tree $T(\phi)$, the sequence $\{a_n\}$ appears as an infinite branch, beginning with the root $a_0$ and proceeding $a_{m-1}, a_{m-2}, \ldots, a_1$ at levels $1, 2, \ldots, m,$ and then repeating.

We now define $\Stab(\phi)$ to be the stabilizer of the action of $\Aut(T(\phi))$ on this branch. It is clear that $G(\phi) \subset \Stab(\phi)$ because the elements of $G(\phi)$ must fix elements of $K$. According to~\cite{jone}, it is hoped that $G(\phi)$ is often a large subgroup of $\Stab(\phi)$; i.e., it is hoped that the index $[\Stab(\phi) : G(\phi)]$ is finite for many possibilities of $\phi$, so that we could take $N(\phi) = \Stab(\phi)$ for these functions $\phi$. One way of checking this is to study the \emph{Hausdorff dimension} of $G(\phi)$ in $\Stab(\phi)$, which is defined as follows.
\begin{defn}[Equation (1) of~\cite{hd}]\label{def1}
Let $T$ be a complete rooted $d$-ary tree of infinite height, and for each $n \in \BN$ let $T_n$ denote the subtree of height $n$ whose root is that of $T$. The \emph{Hausdorff dimension} $\on{hd}(H,G)$ for a subgroup $H$ in a group $G \subset \Aut(T)$ is given by
$$\on{hd}(H, G) = \lim_{n \rightarrow \infty} \frac{\log |H_n|}{\log |G_n|},$$
where for each $n$ we denote by $H_n$ and $G_n$ the quotients of $H$ and $G$, respectively, by the equivalence relation $a \sim b$ if \mbox{$a$ and $b$ act in the same way on $T_n$.}
\end{defn}
To check if a subgroup $H$ of a group $G \subset \Aut(T)$ has infinite index in $G$ (i.e. $[G : H] = \infty$), it suffices to check that $\on{hd}(H, G) < 1$. We now show that, contrary to what might be expected, $\on{hd}(G(\phi),\Stab(\phi)) < 1$, so in fact $[\Stab(\phi) : G(\phi)]$ is necessarily infinite.
\begin{proposition}\label{prop3}
For every $\phi \in K(x)$ such that $\phi^m(0) = 0$ for some positive integer $m$, we have that $\on{hd}(G(\phi),\Stab(\phi)) = 1 - d^{-m}$, and thus $[\Stab(\phi) : G(\phi)]$ is infinite.
\end{proposition}
\begin{proof}
Suppose $n \geq m$. For each $n$, we write $G_n(\phi)$ for the Galois group of the $n^\mathrm{th}$ iterate of $\phi$ and $\Stab_n(\phi)$ for the quotient of $\Stab(\phi)$ by the equivalence relation $a \sim b$ if $a$ and $b$ act in the same way on $T_n(\phi)$. We first compute $|\Stab_n(\phi)|$. Notice that the elements of $\Stab_n(\phi)$ can be enumerated as follows: $\Stab_n(\phi)$ is the product over all $i \in \{1, \dots, n\}$ of the groups $A_i$ of all automorphisms of the forest of $d-1$ subtrees of $T_n(\phi)$ rooted at the children of $a_{i-1}$ other than $a_i$. It is easy to see that $$|A_i| = (d-1)! \cdot \left[(d!)^{\frac{d^{n-i} - 1}{d-1}}\right]^{d-1} = \frac{(d!)^{d^{n-i}}}{d}.$$
We then have that
$$|\Stab_n(\phi)| = \prod_{i = 1}^n |A_i| = \frac{(d!)^{\sum_{i = 1}^n d^{n-i}}}{d^n} = \frac{(d!)^{\frac{d^n - 1}{d-1}}}{d^n}.$$
We next compute an upper bound on $|G_n(\phi)|$. For each $n$, define $S_n(\phi) \subset \Stab_n(\phi)$ to be the group of automorphisms of $T_n(\phi)$ that act in the same way on the subtrees rooted at $a_r$ and $a_s$ whenever $r \equiv s \pmod m$. It is clear that $G_n(\phi) \subset S_n(\phi)$. To show that $\on{hd}(G(\phi),\Stab(\phi)) < 1$, it suffices to show that $\on{hd}(S(\phi),\Stab(\phi)) < 1$, where $S(\phi)$ is the inverse limit of the groups $S_n(\phi)$. We now compute $|S_n(\phi)|$ for each $n$. Noticing that $S_n(\phi)$ is isomorphic to the product $\prod_{i = 1}^m A_i$, we have that
$$|S_n(\phi)| = \prod_{i = 1}^m |A_i| = \frac{(d!)^{\sum_{i = 1}^m d^{n-i}}}{d^m} = \frac{(d!)^{\frac{d^n - d^{n-m}}{d-1}}}{d^m}.$$
We then compute the Hausdorff dimension $\on{hd}(S(\phi),\Stab(\phi))$ in $\Stab(\phi)$ as
\begin{eqnarray*}
\on{hd}(S(\phi),\Stab(\phi)) & = & \lim_{n \rightarrow \infty} \frac{\log |S_n(\phi)|}{\log |\Stab_n(\phi)|} \\
& = & \lim_{n \rightarrow \infty} \frac{(\log d!) \cdot \frac{d^n - d^{n-m}}{d-1} - m}{(\log d!) \cdot \frac{d^n - 1}{d-1} - n} \\
& = & 1 - d^{-m}.
\end{eqnarray*}
Since $m \geq 1, d \geq 2$, we have that $\on{hd}(S(\phi),\Stab(\phi)) =  1 - d^{-m} < 1$, so we have that $\on{hd}(G(\phi),\Stab(\phi)) < 1$. It follows that $[\Stab(\phi) : G(\phi)]$ is necessarily infinite.
\end{proof}
\begin{remark}
It follows from Proposition~\ref{prop3} that $[\Aut(T(\phi)) : G(\phi)] = \infty$ for $\phi$ satisfying the conditions of Problem (1), as stated in Section~\ref{intro}. Also, from the proof of Proposition~\ref{prop3}, one can easily determine that
$$[\Stab_n(\phi) : S_n(\phi)] = \frac{(d!)^{\frac{d^{n-m}-1}{d-1}}}{d^{n-m}},$$
which is trivially bounded below by
$$(d!)^{\frac{d^{n-m}-1}{d-1} - (n-m)}.$$
\end{remark}

\subsection{The Size of $G(\phi)$ as a Subgroup of $S(\phi)$}

\noindent One question that arises from the proof of Proposition~\ref{prop3} is whether $[S(\phi) : G(\phi)]$ is finite, a situation that would allow us to take $N(\phi) = S(\phi)$. Restricting to the case of $K = \BQ$, we provide an example of when $[S(\phi) : G(\phi)]$ is in fact equal to $1$. Notice that \emph{very} few examples exist where one can precisely determine $G(\phi)$; one can refer to Section 2.1 of~\cite{jone} for more discussion on this point.

\begin{example}\label{ex2}
Take $K = \BQ$ and $\phi(x) = x^2 + x \in \BQ(x)$. Then, $m = 1$ and we have $G_n(\phi) \hookrightarrow S_n(\phi) \simeq \Aut(T_n'(\phi)),$ where $T_n'(\phi)$ denotes the subtree of $T_n(\phi)$ rooted at $-1$, the nonzero child of the root of $T_n(\phi)$. Observe that for each $n \geq 1$ we have $G_n(\phi) = \Gal(\BQ(\phi^{-n}(0))/\BQ) = \Gal(\BQ(\phi^{-(n-1)}(-1))/\BQ)$. Now let $\mu$ be the M\"{o}bius transformation defined by $\mu(x) = x - 1$. Then taking $\psi = \mu^{-1} \circ \phi \circ \mu$, we have an equality of Galois groups
\begin{eqnarray*}
\Gal(\BQ(\phi^{-(n-1)}(-1))/\BQ) & = & \Gal(\BQ(\psi^{-(n-1)}(\mu^{-1}(-1)))/\BQ) \\
& = & \Gal(\BQ(\psi^{-(n-1)}(0))/\BQ).
\end{eqnarray*}
But notice that $\psi$ is precisely the polynomial corresponding to Sylvester's sequence, which was studied in~\cite{odo2}. Indeed, we have
$$\psi(x) = [(x-1)^2 + (x-1)] + 1 = x^2 - x + 1.$$
From~\cite{odo2}, we know that for each $n$ we have an isomorphism
$$\Gal(\BQ(\psi^{-(n-1)}(0))/\BQ) \simeq \Aut(T_{n-1}(\psi)) = \Aut(T_n'(\phi)).$$
It follows that for each $n \geq 1$ we have the \emph{equality}, not just an injection,
$$G_n(\phi) = S_n(\phi).$$
Because $G(\phi)$ is the inverse limit of $\{G_n(\phi)\}$ and because $S(\phi)$ is the inverse limit of $\{S_n(\phi)\}$, we have that $G(\phi) = S(\phi)$. Thus, $[S(\phi) : G(\phi)] = 1$.
\end{example}

In the case of $K = \BQ$, the method presented in Example~\ref{ex2} can be used to analyze any map of the form $\phi(x) = x^2+ kx \in \BZ[x]$. We find that the following result, which can be deduced from p.~7 of~\cite{jone} and from Theorem 1.2 of~\cite{jone2}, is useful for studying rational functions of the form $\phi(x) = x^2 + kx$.

\begin{theorem}[\cite{jone2,jone}]\label{thm7}
Let $K = \BQ$, and let $\phi(x) \in \BZ[x]$ be
$$\phi(x) = x^2 - kx + k.$$
Then if $k \in \BZ\setminus \{-2,0,2,4\}$, we have that $[\Aut(T(\phi)) : G(\phi)] < \infty$. Moreover, for all $k, a_0 \in \BZ$, we have that $\on{nd}(\phi, a_0) = 0$.
\end{theorem}

Using the first part of Theorem~\ref{thm7}, we obtain the following result.

\begin{theorem}\label{thm8}
Let $K = \BQ$, and let $\phi(x) \in \BZ[x]$ \mbox{be of the form}
$$\phi(x) = x^2 + kx.$$
Then if $k \in \BZ\setminus \{-2,0,2,4\}$, we have that $[S(\phi) : G(\phi)] < \infty$.
\end{theorem}
\begin{proof}
We have $G_n(\phi) \hookrightarrow S_n(\phi) \simeq \Aut(T_n'(\phi)),$ where $T_n'(\phi)$ denotes the subtree of $T_n(\phi)$ rooted at $-k$, the nonzero child of the root of $T_n(\phi)$. For each $n \geq 1$, we have $G_n(\phi) = \Gal(\BQ(\phi^{-n}(0))/\BQ) = \Gal(\BQ(\phi^{-(n-1)}(-k))/\BQ)$. Let $\mu$ be the M\"{o}bius transformation defined by $\mu(x) = x - k$. Then taking $\psi = \mu^{-1} \circ \phi \circ \mu$ as in Example~\ref{ex2}, we have that $\Gal(\BQ(\phi^{-(n-1)}(-k))/\BQ) =  \Gal(\BQ(\psi^{-(n-1)}(\mu^{-1}(-k)))/\BQ) =\Gal(\BQ(\psi^{-(n-1)}(0))/\BQ)$, from which we deduce that $$G_n(\phi) = G_{n-1}(\psi) \hookrightarrow \Aut(T_{n-1}(\psi)) \simeq \Aut(T_n'(\phi)) \simeq S_n(\phi).$$
Consider the inverse systems $\{G_n(\phi)\} \to G(\phi)$, $\{G_{n-1}(\psi)\}$ $\to G(\psi)$, $\{\Aut(T_{n-1}(\psi))\}$ $\to \Aut(T(\psi))$, and $\{S_n(\phi)\} \to S(\phi)$, and observe that we have the equalities $[G(\psi) : G(\phi)] = [S(\phi) : \Aut(T(\psi))] = 1$. Now suppose $k \notin \{-2,0,2,4\}$. Notice that $\psi(x) = x^2 - kx + k$, so since $k \notin \{-2,0,2,4\}$, Theorem~\ref{thm7} implies that $[\Aut(T(\psi)) : G(\psi)] < \infty$. It follows from the multiplicativity of the group index that $[S(\phi) : G(\phi)] < \infty$.
\end{proof}
It is natural to consider what can be deduced in the case that $k \in \{-2,0,2,4\}$. If $k = 0$, then $G(\phi)$ is trivial, so necessarily we have $[S(\phi) : G(\phi)] = \infty$. The question of whether $[S(\phi) : G(\phi)]$ is finite is more difficult to answer when $k \in \{-2,2,4\}$. The following result handles each of these three cases individually.

\begin{theorem}\label{thm9}
Let $K = \BQ$, and let $\phi(x) \in \BZ[x]$ \mbox{be of the form}
$$\phi(x) = x^2 + kx.$$
Then if $k \in \{-2,2,4\}$, we have that $[S(\phi) : G(\phi)] = \infty$. Specifically:
\begin{enumerate}
\item When $k = -2$, $G(\phi)$ has index $2$ in a subgroup of $\Aut(T(\phi))$ that is isomorphic to $\underset{\leftarrow n}\lim \, (\BZ/(3 \cdot 2^n)\BZ)^*$.
\item When $k = 2$, $G(\phi) \simeq \BZ/2\BZ \times \BZ_2$, where $\BZ_2$ is the ring of $2$-adic integers.
\item When $k = 4$, $G(\phi)$ has index $2$ in a subgroup of $\Aut(T(\phi))$ that is isomorphic to $\BZ/2\BZ \times \BZ_2$.
\end{enumerate}
\end{theorem}
\begin{proof}
Much of the argument used to prove Theorem~\ref{thm8} still applies: our method still consists of reducing the question to computing the Galois action on a subtree of $T(\phi)$, and then employing a direct analysis. Indeed, regardless of our choice of $k$, we have $G_n(\phi) \hookrightarrow S_n(\phi) \simeq \Aut(T_n'(\phi)),$ where $T_n'(\phi)$ denotes the subtree of $T_n(\phi)$ rooted at the ``$-k$'' child of the root of $T_n(\phi)$. For each $n \geq 1$, we have $G_n(\phi) = \Gal(\BQ(\phi^{-n}(0))/\BQ) = \Gal(\BQ(\phi^{-(n-1)}(-k))/\BQ)$.
\begin{enumerate}
\item Suppose $k = -2$. Then $\phi(x) = x^2 - 2x$, and conjugating $\phi$ by the M\"{o}bius transformation $\mu(x) = x +1$, we obtain $\psi(x) = (\mu^{-1}\circ\phi\circ\mu)(x) = x^2 - 2$. We then have the equality of Galois groups $G_n(\phi) = \Gal(\BQ(\phi^{-(n-1)}(2))/\BQ) = \Gal(\BQ(\psi^{-(n-1)}(1))/\BQ)$. Notice that the roots of the equation $\psi^n(x) = 1$ are always real and that $\zeta_{3 \cdot 2^{n+1}} + \zeta_{3 \cdot 2^{n+1}}^{-1}$ is a root of $\psi^n(x) = 1$ for each $n$, where $\zeta_{m}$ is a primitive $m^{\mathrm{th}}$ root of unity. Further observe that we have the inclusions
    $$\BQ(\zeta_{3 \cdot 2^n} + \zeta_{3 \cdot 2^n}^{-1}) \hookrightarrow \BQ(\psi^{-(n-1)}(1))/\BQ \hookrightarrow \BQ(\zeta_{3 \cdot 2^n})$$
    and that the degree of the composite extension is $[\BQ(\zeta_{3 \cdot 2^n}) : \BQ(\zeta_{3 \cdot 2^n} + \zeta_{3 \cdot 2^n}^{-1})] = 2$. Since $\BQ(\zeta_{3 \cdot 2^n})$ is not a real extension of $\BQ$, we have that $\BQ(\zeta_{3 \cdot 2^n} + \zeta_{3 \cdot 2^n}^{-1}) = \BQ(\psi^{-(n-1)}(1))/\BQ$. Then $[\BQ(\zeta_{3 \cdot 2^n}) : \BQ(\psi^{-(n-1)}(1))/\BQ] = 2$, so we find that
   $$|G_n(\phi)| = |\Gal(\BQ(\psi^{-(n-1)}(1))/\BQ)| = \frac{|\Gal(\BQ(\zeta_{3 \cdot 2^n}))|}{2} = \frac{2^n}{2} = 2^{n-1}.$$
   Thus, we can compute the Hausdorff dimension of $G(\phi)$ in $S(\phi)$ to be
    $$\on{hd}(G(\phi), S(\phi)) = \lim_{n \rightarrow \infty} \frac{\log 2^{n-1}}{\log 2^{2^{n-1}-1}} = \lim_{n \rightarrow \infty} \frac{n-1}{2^{n-1} - 1} = 0 < 1,$$
    so indeed $[S(\phi) : G(\phi)] = \infty$. We also have that $G(\phi)$ embeds with index $2$ into the inverse limit of the inverse system $\{\Gal(\BQ(\zeta_{3 \cdot 2^n}))\} \simeq \{(\BZ/(3 \cdot 2^n)\BZ)^*\}$, since for each $n$ we have
    $$\Gal(\BQ(\zeta_{3 \cdot 2^n})) \simeq (\BZ/(3 \cdot 2^n)\BZ)^*.$$
\item Suppose $k = 2$. Then $\phi(x) = x^2 + 2x$, and conjugating $\phi$ by the M\"{o}bius transformation $\mu(x) = x - 1$, we obtain $\psi(x) = (\mu^{-1} \circ \phi \circ \mu)(x) = x^2$. We then have the equality of Galois groups $G_n(\phi) = \Gal(\BQ(\phi^{-(n-1)}(-2))/\BQ) = \Gal(\BQ(\psi^{-(n-1)}(-1))/\BQ)$. Thus, to determine $G_n(\phi)$, it suffices to determine $\Gal(\BQ(\psi^{-(n-1)}(-1))/\BQ)$, which is the Galois group of the polynomial $x^{2^{n-1}} + 1$. Notice that the Galois group of $x^{2^n} - 1$ is $\BZ/2\BZ \times \BZ/2^{n-2}\BZ$ when $n \geq 2$, and a splitting field for both polynomials $x^{2^{n-1}} + 1$ and $x^{2^n} - 1$ is $\BQ(\zeta_{2^n})$. It follows that we have the isomorphism $$G_n(\phi) \simeq \BZ/2\BZ \times \BZ/2^{n-2}\BZ$$
    for each $n$. Thus $|G_n(\phi)| = 2^{n-1}$, and from the proof of Proposition~\ref{prop3} we have that $|S_n(\phi)| = 2^{2^{n-1} -1}$. Thus, we compute the Hausdorff dimension of $G(\phi)$ in $S(\phi)$ to be
    $$\on{hd}(G(\phi), S(\phi)) = \lim_{n \rightarrow \infty} \frac{\log 2^{n-1}}{\log 2^{2^{n-1}-1}} = \lim_{n \rightarrow \infty} \frac{n-1}{2^{n-1} - 1} = 0 < 1,$$
    so indeed $[S(\phi) : G(\phi)] = \infty$. We also have that $G(\phi)$ is the inverse limit of the system $\{G_n(\phi)\} \simeq \{\BZ/2\BZ \times \BZ/2^{n-2}\BZ\}$, which is $\BZ/2\BZ \times \BZ_2$, where $\BZ_2$ is the ring of $2$-adic integers.
\item Suppose $k = 4$. Then $\phi(x) = x^2 + 4x$, and the nonperiodic branch of $T(\phi)$ is rooted at $-4$, the nonzero child of the root of $T(\phi)$. The only child of $-4$ is $-2$, so we have that $G_n(\phi) = \Gal(\BQ(\phi^{-(n-2)}(-2))/\BQ)$. Conjugating $\phi$ by the M\"{o}bius transformation $\nu(x) = x +2$, we obtain $\theta(x) = (\nu^{-1}\phi\nu)(x) = x^2 - 2$. We then have the equality of Galois groups $G_n(\phi) = \Gal(\BQ(\theta^{-(n-2)}(0))/\BQ)$. By an argument analogous to the one used in the proof of the first case above, we find that $|G_n(\phi)| =  2^{n-2}$. Thus, we compute the Hausdorff dimension of $G(\phi)$ in $S(\phi)$ to be
    $$\on{hd}(G(\phi), S(\phi)) = \lim_{n \rightarrow \infty} \frac{\log 2^{n-2}}{\log 2^{2^{n-1}-1}} = \lim_{n \rightarrow \infty} \frac{n-2}{2^{n-1} - 1} = 0 < 1,$$
    so indeed $[S(\phi) : G(\phi)] = \infty$. We also have that $G(\phi)$ embeds with index $2$ into the inverse limit of the system $\{\Gal(\BQ(\zeta_{2^n}))\} \simeq \{\BZ/2\BZ \times \BZ/2^{n-2}\BZ\}$, which is $\BZ/2\BZ \times \BZ_2$.
\end{enumerate}
This concludes the proof of Theorem~\ref{thm9}.
\end{proof}
\begin{remark}
The third case in Theorem~\ref{thm9} serves as a companion to results presented in~\cite{reu}, where the Galois groups of iterates of polynomials $\phi(x) = (x + t)^2 - t - 2$ are studied for $t \in \BZ$ and $|t| > 2$. Indeed, notice that $\phi(x) = (x + t)^2 - t - 2 =  x^2 + 4x$ when $t = 2$.
\end{remark}

It remains to be shown whether the stochastic methods discussed in~\cite{jone} can be extended to obtain zero-density results for $P_\phi(a_0)$ when the forward orbit of $0$ under $\phi$ is periodic. If such a generalization is possible, then Theorems~\ref{thm8} and~\ref{thm9} can be utilized to obtain zero-density results for $P_\phi(a_0)$ in the case of $\phi(x) = x^2 + kx$. Nevertheless, the fact that $x^2 + kx$ is conjugate to $x^2 - kx + k$ taken together with the fact that we already have zero-density results for $P_\phi(a_0)$ when $\phi(x) = x^2 - kx + k$ (see the second part of Theorem~\ref{thm7}) suggests that it may be easy to deduce zero-density results for the family of functions $\phi(x) = x^2 + kx$, where $k \in \BZ$. Indeed, we now state and prove that $\on{nd}(\phi, a_0) = 0$ for this family.

\begin{theorem}\label{thm13}
Let $K = \BQ$, and let $\phi(x) \in \BZ[x]$ be of the form
$$\phi(x) = x^2 + kx.$$
Then for all $a_0 \in \BZ$, we have that $\on{nd}(\phi, a_0) = 0$.
\end{theorem}
\begin{proof}
Observe that $\phi^n(x)$ factorizes over the integers as
$$\phi^n(x) = x(x+k)\prod_{i = 1}^{n-1} (\phi^i(x) + k).$$
Fix $a_0 \in \BZ$. Let $P'_\phi(a_0)$ be the set of all $p \in P_\phi(a_0)$ such that $p \mid \phi^i(a_0) + k$ for some $i \in \BN$, and fix $p \in P'_\phi(a_0)$. Also, let $\mu$ be the M\"{o}bius transformation defined by $\mu(x) = x - k$, and define $\psi \in \BZ[x]$ by $\psi = \mu^{-1} \circ \phi \circ \mu$. Notice that $\psi^i(x) = \phi^i(x-k) + k$. Thus, we have that $p \mid \psi^i(a_0 + k)$, which implies that $p \in P_\psi(a_0 + k)$, and so $P'_\phi(a_0) \subset P_\psi(a_0 + k)$. But, as observed in the proof of Theorem~\ref{thm8}, $\psi(x) = x^2 - kx + k$, so by the second part of Theorem~\ref{thm7}, $P_\psi(a_0 + k)$ has density $0$ in the set of all primes. It follows that $P'_\phi(a_0)$ has density $0$ in the set of all primes.

Now observe that if $p \mid \phi^n(a_0)$ for some $n \in \BN$ but $p \not\in P'_\phi(a_0)$, then $p \mid a_0^2 + ka_0$. Therefore, $P_\phi(a_0) - P'_\phi(a_0)$ is a finite set, so because $P'_\phi(a_0)$ has density $0$ in the set of all primes, we have that $\on{nd}(\phi, a_0) = 0$.
\end{proof}

Theorems~\ref{thm8} and~\ref{thm9} handle the family of functions $\phi(x) = x^2 + kx$ for $k \in \BZ$. A similar process may be used to analyze the family of functions $\phi(x) = x^2 - (k + 1)x + k \in \BZ[x]$, for which $d = 2$ and $m = 2$ (recall with the family $\phi(x) = x^2 + kx$, we had $d = 2$ and $m = 1$). We find that the following result, which can be deduced from p.~7 of~\cite{jone} and from Theorem 1.2 of~\cite{jone2}, is useful for studying rational functions of the form $\phi(x) = x^2 - (k+1)x + k$.

\begin{theorem}[\cite{jone2, jone}]\label{thm10}
Let $K = \BQ$, and let $\phi(x) \in \BZ[x]$ be
$$\phi(x) = x^2  + kx - 1.$$
Then if $k \in \BZ\setminus \{-1,0,2\}$, we have that $[\Aut(T(\phi)) : G(\phi)] < \infty$.
\end{theorem}

Using Theorem~\ref{thm10}, we obtain the following result, which notably bears more restrictions than its counterpart Theorem~\ref{thm8}.

\begin{theorem}\label{thm11}
Let $K = \BQ$, and let $\phi(x) \in \BZ[x]$ be of the form
$$\phi(x) = x^2  - (k+1)x + k.$$
Observing that $G(\phi) \subset G(\phi, 1) \times G(\phi, k+1)$, suppose $[G(\phi, 1) \times G(\phi, k+1) : G(\phi)] < \infty$. Then if $k \in \BZ\setminus \{-1,1,2\}$, we have that $[S(\phi) : G(\phi)] < \infty$.
\end{theorem}
\begin{proof}
In the proof of Proposition~\ref{prop3}, we found that $|S_n(\phi)| = 2^{2^n - 2^{n-2}-2}$. It is in fact difficult to compute $|G_n(\phi)|$ for each $n$. Observe that the vertices of level $1$ are $k$ and $1$ and that the children of $k$ (which live on level $2$) are $0$ and $k + 1$. Taking $n > 2$, we observe that automorphisms of $T_n(\phi)$ that belong to $G_n(\phi)$ are entirely determined by their actions on the complete subtrees rooted at the ``$1$'' vertex of level $1$ and the ``$k+1$'' vertex of level $2$. But what makes $|G_n(\phi)|$ difficult to compute is that perhaps for some $\phi$ we might have $G_n(\phi) \not\simeq G_{n-1}(\phi, 1) \times G_{n-2}(\phi, k+1)$, even though we always have $G_n(\phi) \subset G_{n-1}(\phi, 1) \times G_{n-2}(\phi, k+1)$. In the case where $[G(\phi, 1) \times G(\phi, k+1)  : G(\phi)] < \infty$, we may resolve this problem by simply replacing $G_n(\phi)$ with $G_{n-1}(\phi, 1) \times G_{n-2}(\phi, k+1)$.

Let $G''(\phi) = G(\phi, 1) \times G(\phi, k+1)$, which is the inverse limit of the system $\{G_{n-1}(\phi, 1) \times G_{n-2}(\phi, k+1)\}$. Let $\mu$ be the M\"{o}bius transformation defined by $\mu(x) = x + 1$. Then taking $\psi = \mu^{-1} \circ \phi \circ \mu$, we have an equality of Galois groups $\Gal(\BQ(\phi^{-(n-1)}(1))/\BQ) = \Gal(\BQ(\psi^{-(n-1)}(0))/\BQ)$, from which we deduce that $G_{n-1}(\phi,1) = G_{n-1}(\psi) \hookrightarrow \Aut(T_{n-1}(\psi))$. Denote by $T_n'(\phi)$ the subtree of $T_n(\phi)$ rooted at the ``$1$'' child of the root $0$. Similarly, let $\nu$ be the M\"{o}bius transformation defined by $\nu(x) = x + (k+1)$. Then taking $\pi = \nu^{-1} \circ \phi \circ \nu$, we have an equality of Galois groups
$\Gal(\BQ(\phi^{-(n-2)}(k+1))/\BQ) = \Gal(\BQ(\pi^{-(n-2)}(0))/\BQ)$,
from which we deduce that $G_{n-2}(\phi,k+1) = G_{n-2}(\pi) \hookrightarrow \Aut(T_{n-2}(\pi))$. Denote by $T_n''(\phi)$ the subtree of $T_n(\phi)$ rooted at the ``$k+1$'' child of the level $1$ vertex $k$. Then $\Aut(T_{n-2}(\pi)) \simeq \Aut(T_n''(\phi))$. Now we have that $\Aut(T_n'(\phi)) \times \Aut(T_n''(\phi)) \simeq S_n(\phi)$, so we obtain the embeddings
$$G_{n-1}(\phi,1) \times G_{n-2}(\phi, k_1) = G_{n-1}(\psi) \times G_{n-2}(\pi) \hookrightarrow$$ $$\Aut(T_n'(\phi))\times \Aut(T_n''(\phi)) \simeq S_n(\phi).$$
Let the inverse limit of the system $\{\Aut(T_n'(\phi)) \times \Aut(T_n''(\phi))\}$ be $A'(\phi)$. Then consider the inverse systems
$$\{G_{n-1}(\phi,1) \times G_{n-2}(\phi, k_1)\} \to G''(\phi),  \{G_{n-1}(\psi) \times G_{n-2}(\pi)\} \to G(\psi) \times G(\pi),$$ $$\{\Aut(T_n'(\phi)) \times \Aut(T_n''(\phi))\} \to A'(\phi),  \{S_n(\phi)\} \to S(\phi),$$
and observe that we have the equality $[G(\psi) \times G(\phi): G''(\phi)] = [S(\phi) : A'(\phi)] = 1$. Now suppose $k \notin \{-1,1,2\}$. Notice that $\psi(x) = x^2 + (1-k)x -1$, so since $k \notin \{-1,1,2\}$, Theorem~\ref{thm10} implies that $[A'(\phi) : G(\psi) \times G(\pi)] < \infty$. It follows by \mbox{multiplicativity of the group index that $[S(\phi) : G(\phi)] < \infty$.}
\end{proof}

In light of Theorem~\ref{thm11}, it is again natural to wonder what can be deduced in the case that $k \in \{-1,1,2\}$. The question of whether $[S(\phi) : G(\phi)]$ is finite is more difficult to answer when $k \in \{-1,1,2\}$, because we cannot simply rely on Theorem~\ref{thm10}. Although it is not easy to obtain an analogue of Theorem~\ref{thm9}, we can at least show that the cases of $k = \pm 1$ are equivalent. Indeed, if we denote by $\phi_-$ and $\phi_+$ the polynomials $\phi_-(x) = x^2 - 1$ and $\phi_+(x) = x^2 - 2x + 1$, then one can check that $G(\phi_+) = G(\phi_-) = G(\phi_-, 1)$.

\section{The Size of $C(\phi)$ as a Subgroup of $\Aut(T(\phi))$}\label{trees}

\noindent We now consider the case where $\phi$ commutes with a nontrivial M\"{o}bius transformation that fixes $0$. Let $A(\phi)$ be the group of all such M\"{o}bius transformations (together with the identity). Observe that $A(\phi) \hookrightarrow \Aut(T(\phi))$, and let $C(\phi)$ be the centralizer of $A(\phi)$ in $\Aut(T(\phi))$. Further observe that $G(\phi) \subset C(\phi)$. Since $C(\phi)$ is the obvious candidate for $N(\phi)$ in the case when $A(\phi)$ is nontrivial, we want to better understand $C(\phi)$ as a subgroup of $\Aut(T(\phi))$. In this regard, it is conjectured in~\cite{jone} that the index $[\Aut(T(\phi)) : C(\phi)]$ is infinite.
\begin{conjecture}[Conjecture 3.5 of~\cite{jone}]\label{conj2}
When $A(\phi)$ is nontrivial, $[\Aut(T(\phi)) : C(\phi)] = \infty$.
\end{conjecture}
On p. 19 of~\cite{JM}, Conjecture~\ref{conj2} is proved to be true for quadratic rational functions that have the form
$$\phi(x) = \frac{k(x^2 + 1)}{x} \text{ for } k \in K\setminus\{-1/2,0,1/2\}.$$
The method used in~\cite{JM} is to show that the Hausdorff dimension (recall Definition~\ref{def1}) of $C(\phi)$ in $\Aut(T(\phi))$ is less than $1$. We now modify and extend this method to prove Conjecture~\ref{conj2} in the general case.

Let $T_n$ be a complete rooted $d$-ary tree of height $n$, and for each $i \in \{0, \dots, n\}$ let $T_i$ denote the subtree of height $i$ whose root is that of $T_n$. We extend Proposition 4.1 of~\cite{JM}, which relates the centralizers of some subgroups of $\Aut(T_n)$ to automorphism groups of smaller trees in the case where $d = 2$, as follows.

\begin{proposition}\label{prop1}
Fix $n \in \BN$ and a subgroup $H \subset \Aut(T_n)$, and for each $i \leq n$, let $C_i \subset \Aut(T_i)$ be the centralizer of $H|_{T_i} = \{g|_{T_i} : g \in H\} \subset \Aut(T_i)$. Consider the natural map $C_n \to C_i$ defined by restriction. Then there is an injective homomorphism $$h : \ker(C_n \to C_i) \rightarrow \Aut(T_{n-i})^{m(i)},$$ where $m(i)$ is the number of orbits in the action of $H$ on the vertices of level $i$. Moreover, the injective map $h$ is an isomorphism if the action of $H$ on the vertices of level $i$ is free.\footnote{Recall that the action of a group $G$ on a set $S$ is free if $gx = x \Rightarrow g = 1$ for all $g \in G, x \in X$.}
\end{proposition}
\begin{proof}
Let the vertices of level $i$ be denoted $x_1, \dots, x_{d^i}$. Suppose the action of $H$ on $\{x_1, \dots, x_{d^i}\}$ has orbits $O_1, \dots, O_{m(i)}$, and pick a representative $y_k \in O_k$ for each $k \in \{1, \dots, m(i)\}$. Define the map
$$h : \ker(C_n \to C_i) \rightarrow \Aut(T_{n-i})^{m(i)} \text{ by } \tau \mapsto \left(\tau|_{T_{y_1}}, \dots, \tau|_{T_{y_{m(i)}}}\right),$$
where for each $k \in \{1, \dots, {m(i)}\}$ we denote by $T_{y_k}$ the complete $d$-ary subtree of $T_n$ rooted at $y_k$ (each $T_{y_k}$ has height $n-i$, so each $\tau|_{T_{y_k}}$ can be viewed as an element of $\Aut(T_{n-i})$).

It is clear that $h$ is a homomorphism. We now show that $h$ is injective. Suppose $\tau \in \ker(h)$; i.e., $\tau \in \ker(C_n \to C_i)$ acts by the identity on each of the trees $T_{y_k}$. By definition, $\tau(g(x)) = g(\tau(x))$ for any $g \in H$ and $x \in T_n$, since $\tau \in C_n$. If $x \in T_{y_k}$ for any $k$, then we also have that $\tau(g(x)) = g(x)$. Now suppose $z \in T_{x_j}$ for some $j \in \{1, \dots, d^i\}$, and suppose $x_j \in O_k$. Since the action of $H$ on $O_k$ is transitive, there exists $g \in H$ such that $g(x_j) = y_k$. Observe that $g(z) \in T_{y_k}$, so $\tau(z) = \tau(g^{-1}(g(z))) = g^{-1}(g(z)) = z$. Since $\tau \in \ker(C_n \to C_i)$, we have that $\tau$ also fixes $T_i$. It follows that $\tau = \id$, so $h$ is injective.

Suppose the action of $H$ on the vertices of level $i$ is free. We then show that $h$ is surjective. Take a list of automorphisms $$(\tau_1, \dots, \tau_{m(i)}) \in \Aut(T_{y_1}) \times \dots \times \Aut(T_{y_{m(i)}}) \simeq \Aut(T_{n-i})^{m(i)}.$$ Define $\wt{\tau} \in \Aut(T_n)$ as follows. Because the action of $H$ on the vertices of level $i$ is free, for each $j \in \{1, \dots, d^i\}$ there exists a unique $g_j \in H$ such that $g_j x_j \in \{y_1, \dots, y_{m(i)}\}$. Then let $\wt{\tau}|_{T_{i-1}} = \id$, and if $x_j \in O_\ell$ let $\wt{\tau}|_{T_{x_j}} = g_j^{-1}  \tau_\ell g_j$ for $j \in \{1, \dots, d^i\}$. It is clear that $\wt{\tau}$ is a well-defined tree automorphism, so indeed $\wt{\tau} \in \Aut(T_n)$. We now check that $\wt{\tau} \in C_n$; i.e., we check that $g^{-1}  \wt{\tau} g = \wt{\tau}$ for each $g \in H$. Fix $g \in H$. Then if $v \in T_{i-1}$, we have $\wt{\tau}(v) = v$ and $\wt{\tau}(g(v)) = g(v)$ since $g(v) \in T_{i-1}$, so clearly $(g^{-1}\wt{\tau}g)(v) = \wt{\tau}(v)$. If $v \in T_{x_j}$ for some $j \in \{1, \dots, d^i\}$, then since $g$ is a tree automorphism, we have that $g(v) \in T_{g(x_j)}$. Let $g(x_j) = x_k$. Then, $g_k = g_j g^{-1}$, and so, $(g^{-1}\wt{\tau}g)(v) = (g^{-1}(gg_j^{-1}\tau_\ell g_j g^{-1})g)(v) = (g_j^{-1}\tau_\ell g_j)(v) =\wt{\tau}(v)$, as desired. We finally notice that $\wt{\tau} \in \ker(C_n \to C_i)$ because $\wt{\tau}|_{T_i}$ acts by the identity.
\end{proof}

\begin{remark}
Proposition 4.1 of~\cite{JM} handles the case where $d = 2$, $i = 1$, and $H$ is generated by an involution that swaps the vertices of level $1$.

The isomorphism of Proposition~\ref{prop1} may seem easy to deal with in the case that $H$ acts simply transitively on the vertices of level $i$. In this regard, it is natural to want to describe the subgroups of $\Aut(T_n)$ that act simply transitively on the vertices of level $i$ for some $i$, but there are some obstacles to doing so. Indeed, suppose $H$ acts simply transitively on the vertices of level $i$. Then it is \emph{not necessarily} true that the image of $H$ under the restriction map from $\Aut(T_i)$ to $\Aut(T_{i-1})$ acts simply transitively on the vertices of level $i-1$. As an example, consider the case of $d = 2$ and $i = 3$, and let $x_1, x_2, x_3, x_4$ be the vertices of level $2$. Notice that $\Aut(T_2)$ is a nonabelian group of order $8$ that contains a Klein 4-subgroup. It follows from the classification of groups of order $8$ that $\Aut(T_2)$ is isomorphic to the dihedral group of order $8$. This isomorphism is given explicitly by $\Aut(T_2) \simeq \langle a,x : a^4 = x^2 = 1, xax = a^3\rangle$, where $a$ is the automorphism that sends $(x_1, x_2, x_3, x_4) \mapsto (x_3, x_4, x_2, x_1)$ and $x$ is unique nontrivial element of $\Aut(T_1) \subset \Aut(T_2)$. Now, let $H \subset \Aut(T_3)$ be the subgroup defined by $H \simeq \langle \wt{a},x \rangle$, where $\wt{a}$ is the automorphism that acts by $a$ on $T_2$ and that also swaps the children of each of the nodes $x_1, x_2, x_3, x_4$. Notice that we have the relations $\wt{a}^4 = x^2 = 1, x\wt{a}x = \wt{a}^3$, so $H$ is isomorphic to the dihedral group of order $8$. It is now easy to check that $H \subset \Aut(T_3)$ is a subgroup that acts simply transitively on the vertices of level $3$ but whose restriction to $\Aut(T_2)$ does not act freely on the vertices of level $2$, because the restriction of $H$ to $\Aut(T_2)$ is $\Aut(T_2)$ itself.
\end{remark}

Now let $T$ be a complete rooted $d$-ary tree of infinite height, and for each $n \in \BN$ let $T_n$ denote the subtree of height $n$ whose root is that of $T$. Corollary 4.2 of~\cite{JM} uses the result of Proposition 4.1 of~\cite{JM} to describe the size of the centralizer of a subgroup of $\Aut(T)$, in the case where $d = 2$, $i = 1$, and $H$ is generated by an involution that swaps the vertices of level $1$. The proof of the following lemma does not use the method presented in Corollary 4.2 of~\cite{JM} but still gives an analogous result.

\begin{lemma}\label{cor1}
Let $H \subset \Aut(T)$ be a subgroup, and let $C$ denote the centralizer of $H$ in $\Aut(T)$. Assume the notation of Proposition~\ref{prop1}. Then the Hausdorff dimension of $C$ in $\Aut(T)$ is at most ${m(i)}/d^i$ for each $i \in \BN$, with equality when the \mbox{action of $H$ on the vertices of level $i$ is free.}
\end{lemma}
\begin{proof}
The Hausdorff dimension $\on{hd}(C, \Aut(T))$ of $C$ in $\Aut(T)$ is given by the limit
$$\on{hd}(C,\Aut(T)) = \lim_{n \rightarrow \infty} \frac{\log |C_n|}{\log |\Aut(T_n)|},$$ because $C_n$ is the restriction of $C$ to $\Aut(T_n)$ for each $n$.
Notice that
$$\log |\Aut(T_n)| = \log\left((d!)^{\frac{d^{n}-1}{d-1}}\right) = \frac{d^{n}-1}{d-1} \cdot \log d!.$$
Next, from Proposition~\ref{prop1}, we deduce that $|\ker(C_n \to C_i)| \leq |\Aut(T_{n-i})^{m(i)}|$. We then compute the order of the product of the automorphism groups $\Aut(T_{n-i})^{m(i)}$ to be
$$|\Aut(T_{n-i})^{m(i)}| = (d!)^{{m(i)} \cdot \frac{d^{n-i} - 1}{d-1}}.$$
Thus, we obtain an upper bound on $|C_n|$ since the restriction map $C_n \to C_i$ is a surjection, namely
\begin{eqnarray*}
\log |C_n| & = & \log \big[|\im(C_n \to C_i)| \cdot |\ker(C_n \to C_i)|\big] \\
& \leq & \log |C_i| + \log |\Aut(T_{n-i})^{m(i)}| \\
& = & \log |C_i| + \left({m(i)} \cdot \frac{d^{n-i} - 1}{d-1}\right) \cdot \log d!.
\end{eqnarray*}
Combining our results, we compute an upper bound on the Hausdorff dimension of $C$, namely
\begin{eqnarray*}
\on{hd}(C,\Aut(T)) & = &  \lim_{n \rightarrow \infty} \frac{\log|C_n|}{\log |\Aut(T_n)|} \\
 & \leq & \lim_{n \rightarrow \infty} \frac{\log |C_i| + \left({m(i)} \cdot \frac{d^{n-i} - 1}{d-1}\right)\cdot \log d!}{\frac{d^n-1}{d-1}\cdot \log d!}\\
& = & \frac{{m(i)}}{d^i},
\end{eqnarray*}
which is the desired result.
\end{proof}

\begin{remark}
Lemma~\ref{cor1} is surprising, because the bound on $\on{hd}(C,\Aut(T))$ holds \emph{for each} $i \in \BN$. As a consequence, we have that $m(i+1) = d \cdot m(i)$ when the action of $H$ on the vertices of levels $i$ and $i+1$ is free. This consequence may be readily verified in a given example.
\end{remark}

We now use Lemma~\ref{cor1} to prove Conjecture~\ref{conj2}.

\begin{theorem}\label{thm5}
When $A(\phi)$ is nontrivial we have $[\Aut(T(\phi)) : C(\phi)] = \infty$.
\end{theorem}
\begin{proof}
It follows immediately from Lemma 12 that $\on{hd}(C(\phi), \Aut(T(\phi)) < 1$ if the action of any $\mu \in A(\phi)$ on the tree $T(\phi)$ is nontrivial. But since $\mu$ is a M\"{o}bius transformation, it can have at most two fixed points, so our assumption that $\phi^n$ has $d^n$ distinct roots already implies $\mu$ acts nontrivially on the roots of $\phi^2$.
\end{proof}

The following is an example that illustrates the proof of Theorem~\ref{thm5}.

\begin{example}\label{ex1}
Let $K = \BQ(\zeta_n)$, where $a \in \BQ_{>0}$ and $\zeta_n$ is a primitive $n^\mathrm{th}$ root of $\pm 1$ (for $n > 1$). Consider rational functions $\phi_n(x) \in K(x)$ defined as follows:
$$\phi_n(x) = \pm \frac{k(x^n \mp a)}{x^{n-1}},$$
where $k \in K$. Notice that $\phi_n$ commutes with the M\"{o}bius transformation $m_n$ defined by $m_n(x) = \zeta_nx$, and also observe that $m_n(0) = 0$ for all $n$. Consider the nontrivial subgroup $A'(\phi_n) \subset A(\phi_n)$ generated by $m_n$, and notice that $A'(\phi_n)$ acts simply transitively on level $1$ of $T$. We have that the centralizer $C'(\phi_n)$ of $A'(\phi_n)$ in $\Aut(T(\phi_n))$ contains $C(\phi_n)$. By Lemma~\ref{cor1}, the Hausdorff dimension of $C'(\phi_n)$ is $1/n$, so $[\Aut(T(\phi_n)) : C'(\phi_n)] = \infty$, implying that $[\Aut(T(\phi_n)) : C(\phi_n)] = \infty$.
\end{example}

\section{The Size of $G(\phi)$ as a Subgroup of $C(\phi)$}\label{squares}

\noindent We retain the setting of Section~\ref{trees}; i.e., $\phi$ commutes with a nontrivial M\"{o}bius transformation that fixes $0$. Theorem~\ref{thm5} provides evidence that $C(\phi)$ is a small subgroup of $\Aut(T(\phi))$. But to check that $C(\phi)$ is a good candidate for $N(\phi)$, we want to show that $[C(\phi) : G(\phi)] < \infty$. To further simplify the situation, we restrict to the case of $d = 2$. As shown in Section 2 of~\cite{JM}, it then suffices to consider functions $\phi$ that have the form
$$\phi(x) = \frac{k_0(x^2+1)}{x}.$$
We further restrict to the case of $K = \BQ$ and $k_0 \in \BZ \setminus \{0\}$. In this case, Conjecture 3.9 of~\cite{jone} (alternatively, Conjecture 1.3 of~\cite{JM}) takes the following form.
\begin{conjecture}\label{conj1}
Let $\phi(x) \in \BQ(x)$ be a rational function of the form
$$\phi(x) = \frac{k_0(x^2+1)}{x},$$ where $k_0 \in \BZ \setminus\{0\}$. Then $[C(\phi) : G(\phi)] < \infty$.
\end{conjecture}

The question of resolving Conjecture~\ref{conj1} reduces to a seemingly unrelated question regarding squares in sequences. Let the pair of polynomials $(\delta_n, \varepsilon_n)_{n = 1, 2, \ldots}$ be given by the recursion
\begin{eqnarray*}
(\delta_1(k), \varepsilon_1(k)) & = & (2k^2,k) \text{ and }\\ (\delta_n(k), \varepsilon_n(k)) & = & \left(\delta_{n-1}(k)^2 + \varepsilon_{n-1}(k)^2, \frac{\delta_{n-1}(k)\varepsilon_{n-1}(k)}{k}\right) \text{ for } n \geq 2.
\end{eqnarray*}
(Notice that $\delta_n$ and $\varepsilon_n$ are indeed polynomials for all $n$ because of the values assigned to $\delta_1, \varepsilon_1$.) Jones and Manes then obtain the following result.

\begin{theorem}[Theorem 5.3 of~\cite{JM}]\label{thm1}
Let $\phi(x) \in \BQ(x)$ be a rational function of the form
$$\phi(x) = \frac{k_0(x^2+1)}{x},$$
where $k_0 \in \BZ \setminus\{0\}$. If $\delta_n(k_0)$ is not a square in $\BZ$ for all $n \geq 2$ , then $[C(\phi) : G(\phi)] < \infty$.
\end{theorem}

We therefore make it our objective to prove that $\delta_n(k_0)$ is not a square for all integers $n \geq 2$ and $k_0 \neq 0$. In what follows, we discuss two methods of achieving this objective, namely a congruence method and an algebraic method.

\subsection{Congruence Method}

\noindent In~\cite{JM}, congruence methods are used to make progress toward this objective. Indeed, Jones and Manes obtain the following result by reducing $\delta_n(k_0)$ modulo the primes $2$, $3$, $5$, and $7$.

\begin{theorem}[Theorem 5.8 of~\cite{JM}]\label{thm2}
If one of the following congruences hold, then $\delta_n(k_0)$ is not a square for all $n \geq 2$:
\begin{itemize}
\item $k_0 \equiv 1 \pmod 2$,
\item $k_0 \equiv 1,2 \pmod 3$,
\item $k_0 \equiv 2, 3 \pmod 5$, or
\item $k_0 \equiv 1, 2, 5, 6 \pmod 7$.
\end{itemize}
\end{theorem}
In~\cite{JM} it is shown using such congruence methods that if $\delta_n(k_0)$ is a square in $\BZ$ for some $n,k_0$, then $|k_0| > 10,000$. The method of reducing modulo various primes also results in the following useful theorem.

\begin{theorem}[Theorem 5.8 of~\cite{JM}]\label{thm3}
For all odd $n$ we have that $\delta_n(k_0)$ is not a square for all $k_0 \neq 0$.
\end{theorem}

It is possible to obtain a result analogous to Theorem~\ref{thm2} for every prime, and using an exhaustive search, we can find the congruence classes of $k_0$ modulo a given prime for which $\delta_n(k_0)$ is not a square for all $n \geq 2$. We extend Theorem~\ref{thm2} as follows.
\begin{theorem}\label{thm6}
If one of the following congruences hold, then $\delta_n(k_0)$ is not a square for all $n$. It follows that $\delta_n(k_0)$ is not a square for all $1 \leq |k_0| \leq 10^6$.
\begin{itemize}
\item $k_0 \equiv 1 \pmod 2$,
\item $k_0 \equiv 1,2 \pmod 3$,
\item $k_0 \equiv 2, 3 \pmod 5$,
\item $k_0 \equiv 1, 2, 5, 6 \pmod 7$,
\item $k_0 \equiv 1,2,5,6,9,10 \pmod{11}$,
\item $k_0 \equiv 3,6,7,10 \pmod{13}$,
\item $k_0 \equiv 1,3,14,16 \pmod{17}$,
\item $k_0 \equiv 6,7,9,10,12,13 \pmod{19}$,
\item $k_0 \equiv 1,6,8,9,14,15,17,22\pmod{23}$,
\item $k_0 \equiv 2,11,12,14,15, 17,18,27\pmod{29}$,
\item $k_0 \equiv 6,10,11,14,17,20,21,25\pmod{31}$,
\item $k_0 \equiv 6,8,10,11,14,17,18,19,20,23,26,27,29,31\pmod{37}$,
\item $k_0 \equiv 2,11,13,14,15,26,27,28,30,39\pmod{41}$,
\item $k_0 \equiv 4,5,12,14,17,21,22,26,29,31,38,39\pmod{43}$,
\item $k_0 \equiv 4,7,11,12,15,19,21,26,28,32,35,36,40,43\pmod{47}$,
\item $k_0 \equiv 2,19,22,25,26,27,28,31,34,51\pmod{53}$,
\item $k_0 \equiv 1,3,7,8,51,25,29,30,34,52,56,58\pmod{59}$,
\item $k_0 \equiv 1,2,3,12,17,24,29,30,31,32,37,44,49,58,59,60\pmod{61}$,
\item $k_0 \equiv 5,6,9,15,18,22,27,30,32,33,34,35,37,40,45,49,$\\$52,58,61,62\pmod{67}$,
\item $k_0 \equiv 6,7,11,12,16,20,27,28,30,33,38,41,43,44,51,55,$\\$59,60,64,65\pmod{71}$,
\item $k_0 \equiv 6,7,12,13,14,20,24,29,32,33,40,41,44,49,53,59,$\\$60,61,66,67\pmod{73}$,
\item $k_0 \equiv 20,68\pmod{79}$,
\item $k_0 \equiv 12,16,21,62,71\pmod{83}$,
\item $k_0 \equiv 9,50\pmod{103}$,
\item $k_0 \equiv 106\pmod{107}$,
\item $k_0 \equiv 92 \pmod{109}$,
\item $k_0 \equiv 26,105\pmod{131}$,
\item $k_0 \equiv 89\pmod{149}$,
\item $k_0 \equiv 24\pmod{157}$,
\item $k_0 \equiv 19\pmod{173}$, or
\item $k_0 \equiv 52,145 \pmod{197}$.
\end{itemize}
\end{theorem}
\begin{proof}
For each one of the above pairs of congruence class $k_0$ and prime $p$, one can check that a full period of the (periodic) sequence $(\delta_n(k_0), \varepsilon_n(k_0))$ modulo $p$ is such that $\delta_n(k_0)$ is never a quadratic residue when $n$ is even (by Theorem~\ref{thm3}, we may assume that $n$ is even), implying that $\delta_n(k_0)$ is never a square. In the case that the sequence $(\delta_n(k_0), \varepsilon_n(k_0))$ modulo $p$ is preperiodic (i.e., the sequence behaves nonperiodically before it becomes periodic), one can check that each term of the ``preperiod'' is such that $\delta_n(k_0)$ (when not reduced modulo $p$) is never a square when $n$ is even. It follows from the methods described in~\cite{JM} that $\delta_n(k_0)$ is not a square for all $|k_0| \leq 10^6$ (in fact, these congruences work up to $1,056,575$).
\end{proof}
Theorem~\ref{thm6} provides good evidence that Conjecture~\ref{conj1} is true. Indeed, we would expect that pairs $(n,k_0)$ for which $\delta_n(k_0)$ is a square are such that $n$ and $k_0$ are both small (e.g. recall that the largest square in the Fibonacci sequence is $144$), and Theorem~\ref{thm6} indicates that $\delta_n(k_0)$ is not a square even for small $n$ and $k_0$. By repeating the method described in the proof of Theorem~\ref{thm6}, one can presumably raise the lower bound on $k_0$-values even higher than $10^6$.

\subsection{Algebraic Method}

\noindent We now present a more algebraic study of the polynomials $\delta_n, \varepsilon_n$ (i.e., one that does not involve reducing modulo various primes). To begin with, we state and prove some algebraic properties of the polynomials $\delta_n, \varepsilon_n$ as follows.

\begin{lemma}\label{lem1}
Let $\deg(f)$ denote the degree of a polynomial $f$, and let $\low(f)$ denote the degree of the lowest-degree (nonzero) term of $f$. Then we have the following results:
$$\deg(\delta_n) = 2^n,\quad \low(\delta_n) = \frac{2^n}{3} - \frac{(-1)^n}{3} + 1,$$ $$\deg(\varepsilon_n) = 2^n - n, \quad \text{ and } \quad \low(\varepsilon_n) = \frac{2^n}{3} + \frac{(-1)^n}{6} + \frac{1}{2}.$$
\end{lemma}
\begin{proof}
We first make the following three claims: (1) $\deg(\delta_n) > \deg(\varepsilon_n)$ for all $n$; (2) $\low(\delta_n) = \low(\varepsilon_n)$ for all even $n$; and (3) $\low(\delta_n)  = \low(\varepsilon_n) + 1$ for all odd $n$. Claim (1) clearly holds for $n = 1$. Now suppose Claim (1) holds for some $n \geq 1$. From the equation $\delta_{n+1}(k) = \delta_n(k)^2 + \varepsilon_n(k)^2$, we find that $\deg(\delta_{n+1}) = 2\max(\deg(\delta_n), \deg(\varepsilon_n)) = 2\deg(\delta_n)$ . Also, from the equation $\varepsilon_{n+1}(k) = \frac{\delta_n(k)\varepsilon_n(k)}{k}$, we find that $\deg(\varepsilon_{n+1}) = \deg(\delta_n) + \deg(\varepsilon_n) - 1 < 2\deg(\delta_n) = \deg(\delta_{n+1})$, from which Claim (1) follows by induction.

We prove Claims (2) and (3) simultaneously. These two claims clearly hold for $n = 1$ and $n = 2$. Then suppose the claim holds for $n$ where $n$ is even, and let $a = \low(\delta_n) = \low(\varepsilon_n)$. Notice that $\low(\delta_{n+1}) = 2 \min(\low(\delta_{n}), \low(\varepsilon_{n})) = 2a$
and that $\low(\varepsilon_{n+1}) = \low(\delta_n) + \low(\varepsilon_n) -1 = 2a-1$. Also notice that $\low(\delta_{n+2}) = 2 \min(\low(\delta_{n+1}), \low(\varepsilon_{n+1})) = 4a-2$
and that $\low(\varepsilon_{n+2}) = \low(\delta_{n+1}) + \low(\varepsilon_{n+1}) -1 = 4a - 2$. Claims (2) and (3) then follow by induction.

It follows from Claims (1), (2), and (3) that we have these recursions:
\begin{align*}
\deg(\delta_n) & = 2\deg(\delta_{n-1}),  \deg(\varepsilon_n)  = \deg(\delta_{n-1}) + \deg(\varepsilon_{n-1}) - 1,\\
 \low(\delta_n) & =  2\low(\varepsilon_{n-1}), \low(\varepsilon_n)  = \low(\delta_{n-1}) + \low(\varepsilon_{n-1}) - 1.
\end{align*}
Solving these recursions then yields the lemma.
\end{proof}

\begin{lemma}\label{lem2}
The leading coefficient of $\delta_n$ is a square for all $n \geq 2$. Moreover, the degree of each term in $\delta_n$ is even for all $n$.
\end{lemma}
\begin{proof}
The first statement clearly holds for $n = 2$. From the proof of Lemma~\ref{lem1}, we have that $\deg(\delta_n) > \deg(\varepsilon_n)$ for each $n$, so the leading term of $\delta_{n+1}$ is the square of the leading term of $\delta_n$.

The second statement clearly holds for $n = 1$. From the recursion
\begin{eqnarray*}
(\delta_1(k), \varepsilon_1(k)) & = & (2k^2,k) \text{ and }\\ (\delta_n(k), \varepsilon_n(k)) & = & \left(\delta_{n-1}(k)^2 + \varepsilon_{n-1}(k)^2, \frac{\delta_{n-1}(k)\varepsilon_{n-1}(k)}{k}\right) \text{ for } n \geq 2,
\end{eqnarray*}
we find that $\varepsilon_{n-1}(k)^2 = \delta_{n-1}(k) - \delta_{n-2}(k)^2$,
and solving the recursion in terms of $\delta$'s alone yields that
$$\delta_n(k) = \delta_{n-1}(k)^2 + \frac{\delta_{n-1}(k)\delta_{n-2}(k)^2 - \delta_{n-2}(k)^4}{k^2}.$$
Suppose the second statement holds for all $n \leq m$ where $m \geq 1$, and set $n = m+1$ in the above equation. Then, the right-hand-side of the above equation is easily seen to be such that each term has even degree, so each term of $\delta_{m+1}$ has even degree. Thus, the second statement holds by induction.
\end{proof}

\begin{remark}
It follows from the second statement in Lemma~\ref{lem2} that we need only consider positive $k$ in our search for squares in the sequence $\{\delta_n(k) : n = 1, 2, \ldots\}$.
\end{remark}

\begin{lemma}\label{newlem}
For all even $n \geq 2$, the polynomial $\delta_n$ is not the square of any polynomial (in $\BC[k]$).
\end{lemma}
\begin{proof}
Fix even $n \geq 2$. We claim that $\frac{\low(\delta_n)}{2}$ is odd. Indeed, since we took $n$ to be even, we can write $n = 2m$ for some $m \geq 1$, and so by Lemma~\ref{lem1}, we have that
\begin{eqnarray*}
\frac{\low(\delta_n)}{2} & = & \frac{1}{2}\left(\frac{2^n}{3} - \frac{(-1)^n}{3} + 1\right) \\
& = & \frac{2^{2m-1}}{3} - \frac{1}{6} + \frac{1}{2} \\
& = & \frac{2^{2m-1} +  1}{3} \\
& \equiv & 1 \pmod 2.
\end{eqnarray*}
Now, let $g$ be defined by
$$g(k) = \delta_n(k) \cdot k^{-\low(\delta_n)}.$$
Observe that $g \in \BZ[k]$, and since $\low(\delta_n)$ is even, we have that $\delta_n$ is the square of a polynomial if and only if $g$ is the square of a polynomial. For the sake of contradiction, suppose $g$ is the square of a polynomial $f$. Notice that $\deg f = \frac{\deg g}{2}$ is odd because $\frac{\low(\delta_n)}{2}$ is odd and $\deg \delta_n \equiv 0 \pmod 4$. Let $m$ be the largest nonnegative even integer smaller than $\deg f$ such that the coefficient of $k^m$ in the polynomial $f$ is nonzero (such an $m$ exists because the constant term of $g$ is nonzero, so the constant term of $f$ must also be nonzero), and take this coefficient to be $b$. If $b \neq 0$, then the polynomial $g$ has a nonzero term of odd degree. Since $\low(\delta_n)$ is even, it follows that $\delta_n$ has a nonzero term of odd degree, which contradicts the result of Lemma~\ref{lem2}. It follows that $\delta_n$ is not the square of any polynomial.
\end{proof}

We have now provided a partial characterization of the polynomials $\delta_n$ and $\varepsilon_n$. The following result is a more general statement about polynomials and is not specific to the case of $\delta_n$.

\begin{proposition}\label{newprop}
Let $f \in \BZ[k]$ be a polynomial of positive even degree whose leading coefficient is a square. If $f(k_0)$ is a square for infinitely many integers $k_0$, then $f$ is the square of a polynomial.\footnote{The proof of this proposition was obtained by generalizing an argument presented in~\cite{noam}.}
\end{proposition}
\begin{proof}
Write $f(k) = a^2k^{2n} + \sum_{i = 0}^{2n-1} a_ik^i$, where $a \neq 0$. If we take $g(k) = \sqrt{f(k)}$, then the Taylor expansion of $g(k)$ about the point $k = \infty$ can be expressed as $g(k) = ak^n + \sum_{i = 0}^{n-1} b_ix^i + h(k)$, where the coefficients $b_0, \dots, b_{n-1}$ are rational numbers and where $h(k) = O(1/k)$. Notice that $\lim_{k \to \pm\infty} h(k) = 0$, and suppose for the sake of contradiction that $h$ is not identically $0$. Then for every $\varepsilon > 0$, there exists $M$ such that $0 < |h(k)| < \varepsilon$ whenever $|k| > M$. By taking $\varepsilon$ sufficiently small, we can ensure that $[(b_0 - \varepsilon, b_0) \cup (b_0,b_0 + \varepsilon)] \cap \BZ = \varnothing$, which implies that $b_0 + h(k) \not\in \BZ$ whenever $|k| > M$. Thus, by assumption, there exists an integer $k_0$ so large in absolute value that the following conditions hold: (1) $g(k_0) - b_0 - h(k_0) \in \BZ$, (2) $b_0 + h(k_0) \not\in \BZ$, and (3) $g(k_0) \in \BZ$. These three conditions are clearly contradictory, so we must have that $h$ is identically $0$. It follows that $g$ is a polynomial and that $f(k) = g(k)^2$, so $f$ is indeed the square of a polynomial.
\end{proof}

Combining Proposition~\ref{newprop} with the description of the polynomials $\delta_n$ provided in Lemmas~\ref{lem1},~\ref{lem2}, and~\ref{newlem} yields the following result, which provides evidence toward the conjecture that $\delta_n(k_0)$ is not a square for all integers $n \geq 2$ and $k_0 \neq 0$.

\begin{theorem}\label{thm4}
Fix $n \geq 2$. Then $\delta_n(k_0)$ is a square for at most finitely many nonzero integers $k_0$.
\end{theorem}
\begin{proof}
By Theorem~\ref{thm3}, we may restrict to the case of $n$ even. In this case, Lemma~\ref{newlem} tells us that $\delta_n$ is not the square of a polynomial. By Lemma~\ref{lem1}, $\deg \delta_n$ is positive and even, and by Lemma~\ref{lem2}, the leading coefficient of $\delta_n$ is a square. The theorem now follows immediately from Proposition~\ref{newprop}.
\end{proof}

\begin{remark}
The result of Theorem~\ref{thm4} forms a companion to a result presented on p.~25 of~\cite{JM}, where it is shown that for fixed $k_0$, $\delta_n(k_0)$ is not a square for all but finitely many $n$.
\end{remark}

\section*{Open Problems}

\noindent In the case where the sequence $\{a_n\}$ defined by $a_0 = 0$ and $a_n = \phi(a_{n-1})$ is periodic and $d > 2$, the Galois theoretic properties of iterates of $\phi$ have not yet been studied. Moreover, it remains to be studied for a wider selection of rational functions $\phi$ whether the index $[S(\phi) : G(\phi)]$ is finite. In the case where $\phi$ commutes with a nontrivial M\"{o}bius transformation that fixes $0$, Conjecture~\ref{conj1} remains open; i.e. it is yet to be determined whether $[C(\phi) : G(\phi)]$ is finite for all rational functions $\phi$ of the form $\phi(x) = \frac{k_0(x^2 + 1)}{x}$, and the situation where $d > 2$ has not yet been studied. It would also be interesting to investigate whether one can use standard conjectures, such as the $abc$ conjecture or conjectures on bounds for heights of rational points on curves, to prove that $\delta_n(k_0)$ is not a square. Finally, a number of additional questions and conjectures relating to arboreal Galois representations are posed in~\cite{jone}.

\section*{Acknowledgments}

\noindent This research was supervised by Joe Gallian at the University of Minnesota Duluth REU and is
supported by the National Science Foundation (grant number DMS-1062709) and the
National Security Agency (grant number H98230-11-1-0224). I would like to thank Joe
Gallian for his advising and support. I would also like to thank Rafe Jones for introducing me to the field of arithmetic dynamics, explaining his work, and providing numerous helpful comments on my research. I would like to thank Noam Elkies for many fruitful discussions on polynomials and squarefree sequences, and for directing me to the source~\cite{noam}. I would finally like to thank Adam Hesterberg, Noah Arbesfeld, Simon Rubinstein-Salzedo, and Daniel Kane for helpful discussions and
valuable suggestions on the paper.


\end{document}